\RequirePackage[l2tabu,orthodox]{nag}

\documentclass[11pt,a4paper]{amsart}

\usepackage[T1]{fontenc}
\usepackage{lmodern}
\usepackage[utf8]{inputenc}
\usepackage[english]{babel}
\usepackage{amsmath,amsfonts,amssymb,mathrsfs}
\usepackage{graphicx}
\usepackage{subcaption}
\usepackage{tikz}
\usepackage[all]{xy}
\usepackage{xcolor}
\usepackage[expansion=false]{microtype}
\usepackage{hyperref}

\tikzstyle{vertex}=[circle,fill=white,draw,inner sep=0pt,minimum size=5pt]
\tikzstyle{vortex}=[circle,fill=lightgray,draw,inner sep=0pt,minimum size=5pt]
\tikzstyle{vartex}=[circle,fill=black,draw,inner sep=0pt,minimum size=5pt]

\tikzstyle{bvertex}=[circle,fill=white,draw,inner sep=0pt,minimum size=3pt]
\tikzstyle{bvortex}=[circle,fill=lightgray,draw,inner sep=0pt,minimum size=3pt]
\tikzstyle{bvartex}=[circle,fill=black,draw,inner sep=0pt,minimum size=3pt]
\newcommand{\bvertex}{\node[bvertex]}

\newcommand{\bvartex}{\node[bvartex]}

\numberwithin{equation}{section}
\numberwithin{figure}{section}

\theoremstyle{plain}
\newtheorem{Thm}{Theorem}[section]

\newtheorem{Lemma}[Thm]{Lemma}
\newtheorem{Cor}[Thm]{Corollary}

\newtheorem{ClaimA}{Claim}

\theoremstyle{definition}
\newtheorem{Defn}[Thm]{Definition}

\newtheorem*{Def*}{Definition}

\theoremstyle{remark}
\newtheorem*{Remark}{Remark}

\let \seq=\subseteq

\let \To=\Rightarrow

\let \ToT=\Leftrightarrow
\let \0=\varnothing
\let \1=\infty
\let \al=\alpha
\let \be=\beta

\let \De=\Delta
\let \la=\lambda

\let \si=\sigma

\let \i=\iota
\let \t=\theta

\let \f=\varphi
\let \p=\ldots
\let \h=\text

\let \cd=\cdots

\let \cov=\vartriangleleft

\newcommand{\Th}{Theorem}

\newcommand{\Co}{Corollary}

\newcommand{\Def}{Definition}
\newcommand{\Defs}{Definitions}

\newcommand{\id}{\mathrm{id}}

\newcommand{\J}{\mathfrak{I}}

\DeclareMathOperator{\inv}{inv}

\newcommand{\Pp}{P_{\leq p}}
\newcommand{\Ptp}{P^\t_{\leq p}}

\newcommand{\pta}{\cov_\t}
\newcommand{\Mp}{\mathrm{M}^\t}

\newcommand{\I}{I_n}
\newcommand{\IB}{I_n^B}
\newcommand{\FB}{F_n^B}

\DeclareMathOperator{\ct}{ct}
\DeclareMathOperator{\di}{di}
\DeclareMathOperator{\ci}{ci}
\DeclareMathOperator{\Bci}{Bci}
\DeclareMathOperator{\mct}{mct}
\DeclareMathOperator{\Bct}{Bct}
\DeclareMathOperator{\mBct}{mBct}
\DeclareMathOperator{\Bcv}{Bcv}
\DeclareMathOperator{\cv}{cv}
\DeclareMathOperator{\negg}{neg}
\DeclareMathOperator{\dna}{dna}

\title{Fixed elements of pircon automorphisms}
\author{Mikael Hansson \and Vincent Umutabazi}
\thanks{}
\address{Department of Mathematics, KTH Royal Institute of Technology, SE-100 44 Stockholm, Sweden}
\email{hansson\_mikael@hotmail.com}
\address{Department of Mathematics, College of Science and Technology, University of Rwanda, Po Box: 3900 Kigali-Rwanda}
\email{v.umutabazi@ur.ac.rw}

\begin{document}

\begin{abstract}
We prove that the subposet induced by the fixed elements of any automorphism of a pircon is also a pircon.
By \cite[\Th~6.4]{A-H-H_pircons}, the order complex of any open interval in a pircon is a PL ball or a PL sphere.
We apply our main results to symmetric groups of the form $S_{2n}$.
A consequence is that the fixed point free signed involutions form a pircon under the dual of the Bruhat order on the hyperoctahedral group.
Finally, we prove that this poset is, in fact, EL-shellable, which is a type~$B$ analogue of \cite[\Th~1]{C-C-T}.
\end{abstract}

\maketitle

\section{Introduction} \label{sec:intro}

A pircon is a poset in which every non-trivial principal order ideal is finite and admits a special partial matching (SPM).
The notion of SPMs was introduced in \cite{A-H} as a combinatorial tool for studying the Kazhdan-Lusztig-Vogan polynomials for fixed point free involutions.
Later in \cite{A-H-H_pircons}, pircons were introduced as a generalisation of zircons, which were originally invented by Marietti in \cite{Marietti_zircons}.
Pircons have since been studied in different settings of Kazhdan-Lusztig theory like \cite{Marietti_pircons,C-M_pircons}.

The Bruhat order on any Coxeter group is a zircon.
More generally, if $\phi$ is an involutive automorphism of a Coxeter system $(W,S)$,
then the induced Bruhat order on the set $\J(\phi)=\{w \in W \mid \phi(w)=w^{-1}\}$ of twisted involutions is a zircon \cite{Hultman2}.
Examples of pircons include the Bruhat order on parabolic quotients $W^J$, $J \seq S$, and on the set $\i(\phi)=\{\phi(w)w^{-1} \mid w \in W\}$ of twisted identities, whenever $W$ is finite and not of type~$A_{2n}$ \cite{A-H-H_pircons}.

The order complex of any open interval in a zircon is a PL sphere \cite{Marietti_zircons}.
In a pircon, every open interval is a PL ball or a PL sphere \cite{A-H-H_pircons}.
In \cite{Hultman_zircon}, it was proved that for a zircon $P$ with any automorphism $\t$, the subposet $P^\t$ induced by the fixed elements of $\t$ is itself a zircon.

In this paper, the main results are as follows.
It is first proved that if a finite poset $P$ has a maximum and an SPM, then $P^\t$, where $\t$ is any automorphism of $P$, also admits an SPM.
Then, it is proved that if $P$ is a pircon, then so is $P^\t$.
As an application, where $P$ is taken to be the dual of the Bruhat order on the fixed point free involutions in the symmetric group, we deduce that the fixed point free signed involutions form a pircon.
Therefore, by \cite[\Th~6.4]{A-H-H_pircons}, the order complex of (the proper part of) this poset is a PL ball, which is a type~$B$ analogue of a result of Can, Cherniavsky, and Twelbeck \cite{C-C-T}.
They actually proved that the fixed point free involutions are EL-shellable.
We extend this result by proving that the fixed point free signed involutions are EL-shellable, too.

The rest of this paper is organised as follows.
In Section~\ref{sec:prel}, we recall some definitions and properties of SPMs, pircons, and shellability.
Section~\ref{sec:main} contains the main results, where we generalise the main results in \cite{Hultman_zircon}.
In Section~\ref{sec:app}, we apply our main results to fixed point free involutions.
Finally, in Section~\ref{sec:EL}, we establish EL-shellability of the fixed point free signed involutions.

\section{Preliminaries} \label{sec:prel}

In this section, we review some preliminaries needed in the sequel.

\subsection{Posets and pircons}

Let $P$ and $Q$ be two posets (partially ordered sets).
Recall that a map $\t \colon P \to Q$ is \emph{order-preserving} if for all $x,y \in P$, $x \leq y$ in $P$ implies $\t(x) \leq \t(y)$ in $Q$.
An \emph{isomorphism} is a bijective order-preserving map whose inverse is also order-preserving.
If $P=Q$, $\t$ is called an \emph{automorphism} of $P$.
In this case, let $P^\t$ be the subposet of $P$ induced by the fixed elements of $\t$.

Suppose that $x,y \in P$ with $x<y$.
We say that $x$ is \emph{covered} by $y$, and write $x \cov y$, if there is no $z \in P$ such that $x<z<y$.
An \emph{order ideal} $I \seq P$ is an induced subposet of $P$ such that, if $x \leq y$ and $y \in I$, then $x \in I$.
An order ideal with a \emph{maximum} (i.e., an element $\hat{1}$ such that $x \leq \hat{1}$ for all $x \in P$) is called \emph{principal}.
In particular, let $\Pp=\{x \in P \mid x \leq p\}$.

A \emph{matching} on $P$ is an involution $M \colon P \to P$ such that $x \cov M(x)$ or $M(x) \cov x$ for all $x \in P$.
As invented by Brenti \cite{Brenti_history,Brenti_intersection}, a matching $M$ on $P$ is \emph{special} if for all $x,y \in P$ with $x \cov y$, either $M(x)=y$ or $M(x)<M(y)$.
When $P$ is Eulerian, special matchings are similar to \emph{compression labellings} introduced by du Cloux \cite{du_Cloux}.

The following definition is taken from \cite{Hultman_zircon}.
Zircons were originally defined by Marietti \cite{Marietti_zircons} in a different way, but the two definitions are equivalent as proved in \cite[Proposition 2.3]{Hultman_zircon}.

\begin{Defn}[\cite{Hultman_zircon}] \label{zircon}
A poset $P$ is a \emph{zircon} if, for every non-minimal element $p \in P$, the principal order ideal $\Pp$ is finite and admits a special matching.
\end{Defn}

The next definition generalises special matchings.
Indeed, we note that a special matching is an SPM without fixed elements.

\begin{Defn}[\cite{A-H}] \label{spm}
Let $P$ be a finite poset with $\hat{1}$ and covering relation $\cov$.
A \emph{special partial matching}, or \emph{SPM}, on $P$ is a function $M \colon P \to P$ such that
\begin{itemize}
  \item $M^2=\id$,
  \item $M(\hat{1}) \cov \hat{1}$,
  \item for all $x \in P$, we have $M(x) \cov x$, $x \cov M(x)$, or $M(x)=x$, and
  \item if $x \cov y$ and $M(x) \neq y$, then $M(x)<M(y)$.
\end{itemize}
\end{Defn}

The following lemma is the ``lifting property'' for SPMs.
It will specifically serve as the main tool in the proof of \Th~\ref{thm:autom}.

\begin{Lemma}[\cite{A-H-H_pircons}, Lifting property] \label{lifting}
Suppose that $P$ is a finite poset with $\hat{1}$ and an SPM $M$. If $x,y \in P$ with $x<y$ and $M(y) \leq y$, then
\begin{itemize}
  \item[(i)]   $M(x) \leq y$,
  \item[(ii)]  $M(x) \leq x \To M(x)<M(y)$, and
  \item[(iii)] $M(x) \geq x \To x \leq M(y)$.
 \end{itemize}
\end{Lemma}

We now have the following definition.

\begin{Defn}[\cite{A-H-H_pircons}] \label{pircon}
A poset $P$ is a \emph{pircon} if, for every non-minimal element $p \in P$, the principal order ideal $\Pp$ is finite and admits an SPM.
\end{Defn}

From \Defs~\ref{zircon} and \ref{pircon}, it is clear that every zircon is a pircon.
Pircons are the main objects of study in this paper.

\subsection{Shellability and signed permutations}

A finite poset is \emph{bounded} if it has a minimum and a maximum, and \emph{graded} if every maximal chain has the same length.
A chain $x_0<x_1<\cd<x_k$ is \emph{saturated} if $x_{i-1} \cov x_i$ for all $i \in [k]=\{1,2,\p,k\}$, and an \emph{$x$-$y$-chain} is a saturated chain from $x$ to $y$.

Let $P$ be a finite, bounded, and graded poset.
An \emph{edge-labelling} of $P$ is a function $\la \colon \{(x,y) \in P^2 \mid x \cov y\} \to Q$, where $Q$ is some totally ordered set.
If $\la$ is an edge-labelling of $P$ and $C$ is an $x_0$-$x_k$-chain, let $\la(C)=(\la(x_0,x_1),\la(x_1,x_2),\p,\la(x_{k-1},x_k))$.
The chain is called \emph{increasing} if the sequence $\la(C)$ is weakly increasing, and \emph{decreasing} if $\la(C)$ is strongly decreasing.
An edge-labelling $\la$ of $P$ is an \emph{EL-labelling} if, for all $x<y$ in $P$, there is exactly one increasing $x$-$y$-chain and this chain is \emph{lex-minimal} (i.e., minimal in the lexicographic order) among the $x$-$y$-chains in $P$.
If $P$ has an EL-labelling, $P$ is called \emph{EL-shellable}, because then its order complex $\De(P)$ is shellable~\cite{Bjorner}.

Let $S_{2n}$ denote the group of permutations of $[\pm n]=\{\pm 1,\pm 2,\p,\pm n\}$.
With the adjacent transpositions $(i,i+1)$ as generators, this is a Coxeter group of type~$A_{2n-1}$.
Consider the subgroup $S_n^B$ of permutations $\si$ of $[\pm n]$ such that $\si(-i)=-\si(i)$ for all $i \in [\pm n]$.
This is a standard way to represent the hyperoctahedral group $B_n$ as the group of ``signed permutations''.
(The Coxeter generators are $(1,-1)$ and $(i,i+1)(-i,-i-1)$ for $i \in [n-1]$.)
We shall use $\I$ for the involutions in $S_n$, $\IB$ for the involutions in $S_n^B$ (the ``signed involutions''), and $\FB$ for the fixed point free signed involutions.

The Bruhat order on $S_n^B$ is an induced subposet of the Bruhat order on $S_{2n}$.
It may be defined as follows, where $\si[i,j]=|\{k \leq i \mid \si(k) \geq j\}|$ (see, e.g., \cite[\Th~8.1.8]{B-B}).

\begin{Defn} \label{order}
Let $\si,\tau \in S_n^B$. Then $\si \leq \tau$ if and only if $\si[i,j] \leq \tau[i,j]$ for all $(i,j) \in [\pm n]^2$.
\end{Defn}

\section{Pircons and automorphisms} \label{sec:main}

The following theorem and its corollary generalise the main results of \cite{Hultman_zircon} from special matchings to special partial matchings.
The proof ideas are similar, but the possibility of $M$ fixing elements introduces additional complications.

\begin{Thm} \label{thm:autom}
Let $P$ be a finite poset with a maximum and a special partial matching $M$.
Let also $\t$ be any automorphism of $P$.
Then the subposet $P^\t$ induced by the fixed elements of $\t$ has a special partial matching.
\end{Thm}

\begin{proof}
Since $P$ is finite, $\t$ has finite order $K$.
Observe that each automorphism $\t^i$, $i \in [K]$, transforms the SPM $M$ into an SPM $M_i$, i.e., $\t^i \circ M=M_i \circ \t^i$.
In particular, $M_K=M$.

For a given $p \in P$, define
\[
C(p)=\{a \in P \mid \h{$a=M_{i_t} \circ M_{i_{t-1}} \circ \cd \circ M_{i_1}(p)$ for some $i_1,\p,i_t \in [K]$}\}.
\]
Thus, $C(p)$ consists of the elements that are connected to $p$ by the SPMs $M_i$.
By abuse of notation, we also let $C(p)$ denote the subposet of $P$ induced by the set $C(p)$.
For $a \in C(p)$ given by $a=M_{i_t} \circ M_{i_{t-1}} \circ \cd \circ M_{i_1}(p)$,
define $b \in C(p)$ as $b=M_{i_t}' \circ M_{i_{t-1}}' \circ \cd \circ M_{i_1}'(p)$, where the $M_{i_j}'$ are recursively given by
\begin{equation} \label{prim}
M_{i_j}'=\begin{cases}M_{i_j} & \h{if $M_{i_j} \circ M_{i_{j-1}}' \circ \cd \circ M_{i_1}'(p)<M_{i_{j-1}}' \circ \cd \circ M_{i_1}'(p)$,} \\
\id & \h{otherwise.}\end{cases}
\end{equation}

For $a_j,b_j \in C(p)$, where $a_j=M_{i_j} \circ \cd \circ M_{i_1}(p)$ and $b_j=M_{i_j}' \circ \cd \circ M_{i_1}'(p)$, $1 \leq j \leq t$, we have $b_j \leq a_j$.
To see it, we use induction by assuming that $b_{j-1} \leq b_{a-1}$.
We have the following three cases.
\begin{itemize}
  \item[(i)]   If $a_j \geq a_{j-1}$, there is nothing to prove.
  \item[(ii)]  If $a_j<a_{j-1}$ and $M_{i_j}'=M_{i_j}$, we have $b_j= M_{i_j}(b_{j-1})<b_{j-1}$ and $M_{i_j}(a_{j-1})=a_j<a_{j-1}$.
Hence, we may apply the lifting property to get $b_j \leq a_j$ (with equality if and only if $a_{j-1}=b_{j-1}$).
  \item[(iii)] Finally, if $a_j<a_{j-1}$ and $M_{i_j}'=\id$, we have $M_{i_j}(b_{j-1}) \geq b_{j-1}$ and $M_{i_j}(a_{j-1})=a_j<a_{j-1}$.
We can then apply the lifting property to get $b_j=b_{j-1} \leq a_j$.
\end{itemize}

Thus $b_j \leq a_j$, proving the claim. In particular, $b \leq a$.
By construction, $b \leq p$, so if $p$ and $a$ are both minimal in $C(p)$, we have $a=b=p$.
Hence $C(p)$ contains a unique minimal element.
Analogously, if we reverse the inequality in \eqref{prim}, $C(p)$ also contains a unique maximal element.%
\footnote{Observe that $P$ need not have a minimum.
However, Lemma~\ref{lifting} will still hold, with the same proof as in \cite{A-H-H_pircons}, even if $P$ does not have a maximum, as long as $M$ satisfies the other three conditions in \Def~\ref{spm}.}

If $p$ is not minimal in $C(p)$, then $M_i(p) \cov p$ for some $i$.
To see this, note that if $a<p$ in the above construction, then $b \neq p$ and so we may choose $i=i_j$ for the minimal $j$ which satisfies $M_{i_j}'=M_{i_j}$.
Reversing the inequality in \eqref{prim}, a similar reasoning shows that if $p$ is not maximal in $C(p)$, then there is at least one $i$ such that $p \cov M_i(p)$.

If $p \in P^\t$ and $p \leq M_i(p)$ for at least one $i$, then this will hold for all $i$ and hence $p=\min C(p)$.
Similarly, if $p \in P^\t$ and $M_i(p) \leq p$, then $p=\max C(p)$.
Therefore, for any $p \in P^\t$ we have $p=\min C(p) $ or $p=\max C(p)$, both if $C(p)=\{p\}$.

Since $\t$ permutes the SPMs $M_i$, we have $\t(C(p))=C(\t(p))$ for all $p \in P$. Thus the claim below follows.

\begin{ClaimA} \label{claim:c1}
If $p \in P^\t$, then both $\min C(p)$ and $\max C(p)$ belong to $P^\t$.
\end{ClaimA}

Moreover we have:

\begin{ClaimA} \label{claim:c2}
For any $x,y \in P$, if $\min C(x) \leq \max C(y)$, then $\min C(x) \leq \min C(y)$ and $\max C(x) \leq \max C(y)$.
\end{ClaimA}

\begin{proof}
Let $p \in P$ be arbitrary. If $s=\max C(s)$ and $C(s) \neq C(p)$, it follows from the lifting property that $s>p \To s>M_i(p)$ for all $i$.
This shows that $s>p \To s>\max C(p)$ (by repeated application of the lifting property).
A similar argument shows that if $s=\min C(s)$ and $C(s) \neq C(p)$, then $s<p \To s<\min C(p)$.
Hence, by using the above argument, if we take $s=\max C(y)$ and $p=\min C(x)$ where $s>p$, we get $\max C(y) \geq \max C(x)$.
Similarly, if $s=\min C(x)$ and $p=\max C(y)$ where $s<p$, then $\min C(x) \leq \min C(y)$.
\end{proof}

Let $\pta$ denote the covering relation in $P^\t$.

\begin{ClaimA} \label{claim:c3}
If $\al=\min C(p)$ and $\be=\max C(p)$ belong to $P^\t$, then $\al \pta \be$ or $\al=\be$.
\end{ClaimA}

\begin{proof}
Assume that $x\in P^\t$. If $x=\min C(x)$ and $x<\be$, then by Claim~\ref{claim:c2}, $x \leq \al$.
Similarly, if $x=\max C(x)$ and $\al<x$, then by Claim~\ref{claim:c2}, $\be \leq x$.
In either case, $x$ does not satisfy $\al<x<\be$.
\end{proof}

Define the function $\Mp \colon P^\t\to P^\t$ by
\[
\Mp(p)=\begin{cases}\max C(p) & \h{if $p=\min C(p)$,} \\ \min C(p) & \h{if $p=\max C(p)$.}\end{cases}
\]
By Claim~\ref{claim:c1}, $\Mp$ is well defined.
We show that $\Mp$ is an SPM on $P^\t$ by verifying the conditions in \Def~\ref{spm}.

\begin{itemize}
  \item For every $x \in P^\t$, $(\Mp)^2(x)=x$. That is, $(\Mp)^2=\id$.

  \item Since $\t$ is an automorphism of $P$, $\t(\hat{1})=\hat{1}$, so $\hat{1} \in P^\t$.
Because $M_i$ is an SPM, $M_i(\hat{1}) \neq \hat{1}$, so $\hat{1} \neq \min C(\hat{1})$.
By Claim~\ref{claim:c3}, $\Mp(\hat{1})=\min C(\hat{1}) \pta \hat{1}$.

  \item For any $x \in P^\t$, Claim~\ref{claim:c3} shows that $\Mp(x) \pta x$, $x \pta \Mp(x)$, or $\Mp(x)=x$.

  \item Let $x \pta y$ and $\Mp(x) \neq y$. We must show that $\Mp(x)<\Mp(y)$.
By Claim~\ref{claim:c2}, $\min C(x) \leq \min C(y)$ and $\max C(x) \leq \max C(y)$.
Moreover, both inequalities are strict since $\Mp(x) \neq y$.
\begin{itemize}
  \item[(i)]   If $x \neq \min C(x)$, then $\Mp(x)=\min C(x)<\min C(y) \leq \Mp(y)$.
  \item[(ii)]  If $y \neq \max C(y)$, we have $\Mp(x) \leq \max C(x)<\max C(y)=\Mp(y)$.
  \item[(iii)] If $x=\min C(x)$ and $y=\max C(y)$, we have $x=\min C(x)<\min C(y) \leq y$ and $x \leq \max C(x)<\max C(y)=y$.
Because $x \pta y$, we conclude that $\min C(x)=\max C(x)$ and $\min C(y)=\max C(y)$.
Hence $\Mp(x)=x \pta y=\Mp(y)$. \qedhere
\end{itemize}
\end{itemize}
\end{proof}

From \Th~\ref{thm:autom} we have the following corollary.

\begin{Cor} \label{cor:autom}
Let $P$ be a pircon. If $\t$ is any automorphism of $P$, then the subposet $P^\t$ is a pircon.
\end{Cor}

\begin{proof}
For any non-minimal element $p \in P$, the order ideal $\Pp$ is finite and has an SPM.
Let $p \in P^\t$ be non-minimal.
The elements in $\Ptp$ are the fixed elements of the restriction of $\t$ to $\Pp$.
Therefore, by \Th~\ref{thm:autom}, \smash{$\Ptp$} has an SPM.
\end{proof}

\section{An application of \Co~\ref{cor:autom}} \label{sec:app}

As an application of \Co~\ref{cor:autom}, we shall derive an example of a pircon from the fixed point free involutions in the symmetric group $S_{2n}$.
Below we shall freely use Coxeter group theoretic terminology.
For definitions and preliminaries we recommend the reader to consult \cite{B-B}.

Let $w_0$ be the reverse permutation in $S_{2n}$, and $C(w_0)$ the conjugacy class of $w_0$, i.e., the set of fixed point free involutions.
(Recall that $S_{2n}$ denotes the group of permutations of $[\pm n]$.)
By using \cite[\Th~4.3]{A-H}, one can show that $C(w_0)$, with the dual of the Bruhat order inherited from $S_{2n}$, is a pircon.
Consider the Coxeter group, and Bruhat order, automorphism $\f \colon S_{2n} \to S_{2n}$ given by $\si \mapsto w_0\si w_0$.
Observe that the fixed element subgroup $S_{2n}^\f=S_n^B$.
Because $\f$ preserves $C(w_0)$ (i.e., $\f(C(w_0))=C(w_0)$), it follows from \Co~\ref{cor:autom} that $C(w_0)^\f$, i.e., the set $\FB$ of fixed point free signed involutions, is a pircon (where we have identified $\f$ with its restriction to $C(w_0)$).

Since the partial order induced by the Bruhat order on $S_{2n}$ coincides with the Bruhat order on $S_n^B$, we conclude that $\FB$, ordered by the dual of the Bruhat order on $S_n^B$, is a pircon.
By \cite[\Th~6.4]{A-H-H_pircons}, we therefore have the following result.
It is a type~$B$ analogue of \cite[\Th~6]{C-C-T}, which asserts that (the proper part of) $C(w_0)$, with the dual of the Bruhat order inherited from $S_{2n}$, is a PL ball.

\begin{Cor} \label{cor:ball}
The order complex of (the proper part of) the fixed point free signed involutions is a PL ball.
\end{Cor}

\section{EL-shellability of the fixed point free signed involutions} \label{sec:EL}

In this section, we establish a stronger property of the fixed point free signed involutions, namely, EL-shellability.
This is a type~$B$ analogue of \cite[\Th~1]{C-C-T}, which claims EL-shellability of the fixed point free involutions.%
\footnote{It was as a consequence of this result, together with results from \cite{Bjorner,D-K,Hultman3,Incitti_A},
that Can, Cherniavsky, and Twelbeck proved the type~$A$ version of \Co~\ref{cor:ball}.}
We follow the strategy in the proof of \cite[\Th~1.4]{Hansson_graded_A}, where this fact was reproved.
Both proofs rely heavily on Incitti's classification of the covering relation of the Bruhat order on the involutions~\cite{Incitti_A}.
Similarly, our proof here makes heavy us of his classification for the signed involutions~\cite{Incitti_B}.
We shall describe the parts of these classifications that are needed in order to understand the proof given here.
Let us first describe the possible coverings for the involutions, which all appear for the signed involutions.

For $\si,\tau \in \I$, we write $\tau=\ct_{(i,j)}(\si)$, where $\ct$ stands for \emph{covering transformation}, if they agree everywhere but for the elements whose positions are marked with dots in Figure~\ref{Incitti}, which is taken from \cite[Table~1]{Incitti_A}.

\begin{figure}[tb]
\begin{subfigure}[b]{0.3\textwidth}
\begin{tikzpicture}[scale=0.55]
  \draw (0,0) rectangle (5,5);
  \draw[fill=lightgray] (2,2)--(3,2)--(3,3)--(2,3)--(2,2);
  \draw[dotted] (0,0)--(5,5);
  \bvartex (s1) at (2,2) {}; \bvertex (v1) at (3,2) {}; \bvartex (s2) at (3,3) {}; \bvertex (v2) at (2,3) {};
  \draw (2,0.15)--(2,-0.15); \draw (3,0.15)--(3,-0.15);
  \draw (0.15,2)--(-0.15,2); \draw (0.15,3)--(-0.15,3);
  \node [below] at (2,0) {\tiny $i$}; \node [below] at (3,0) {\tiny $j$};
  \node [left]  at (0,2) {\tiny $i$}; \node [left]  at (0,3) {\tiny $j$}; \node [left] at (0,0) {\phantom{\tiny $\si(j)$}};
\end{tikzpicture}
\caption*{Type~1}
\bigskip
\end{subfigure}
\begin{subfigure}[b]{0.3\textwidth}
\begin{tikzpicture}[scale=0.55]
  \draw (0,0) rectangle (5,5);
  \draw[fill=lightgray] (1.5,1.5)--(3.5,1.5)--(3.5,2.5)--(2.5,2.5)--(2.5,3.5)--(1.5,3.5)--(1.5,1.5);
  \draw[dotted] (0,0)--(5,5);
  \bvartex (s1) at (1.5,1.5) {}; \bvertex (v1) at (3.5,1.5) {}; \bvartex (s2) at (3.5,2.5) {};
  \bvertex (v2) at (2.5,2.5) {}; \bvartex (s3) at (2.5,3.5) {}; \bvertex (v3) at (1.5,3.5) {};
  \draw (1.5,0.15)--(1.5,-0.15); \draw (2.5,0.15)--(2.5,-0.15); \draw (3.5,0.15)--(3.5,-0.15);
  \draw (0.15,1.5)--(-0.15,1.5); \draw (0.15,2.5)--(-0.15,2.5); \draw (0.15,3.5)--(-0.15,3.5);
  \node [below] at (1.5,0) {\tiny $i$}; \node [below] at (2.5,0) {\tiny $j$}; \node [below] at (3.5,0) {\tiny $\si(j)$};
  \node [left]  at (0,1.5) {\tiny $i$}; \node [left]  at (0,2.5) {\tiny $j$}; \node [left]  at (0,3.5) {\tiny $\si(j)$};
\end{tikzpicture}
\caption*{Type~2}
\bigskip
\end{subfigure}
\begin{subfigure}[b]{0.3\textwidth}
\begin{tikzpicture}[scale=0.55]
  \draw (0,0) rectangle (5,5);
  \draw[fill=lightgray] (1.5,2.5)--(2.5,2.5)--(2.5,1.5)--(3.5,1.5)--(3.5,3.5)--(1.5,3.5)--(1.5,2.5);
  \draw[dotted] (0,0)--(5,5);
  \bvartex (s1) at (1.5,2.5) {}; \bvertex (v1) at (2.5,2.5) {}; \bvartex (s2) at (2.5,1.5) {};
  \bvertex (v2) at (3.5,1.5) {}; \bvartex (s3) at (3.5,3.5) {}; \bvertex (v3) at (1.5,3.5) {};
  \draw (1.5,0.15)--(1.5,-0.15); \draw (2.5,0.15)--(2.5,-0.15); \draw (3.5,0.15)--(3.5,-0.15);
  \draw (0.15,1.5)--(-0.15,1.5); \draw (0.15,2.5)--(-0.15,2.5); \draw (0.15,3.5)--(-0.15,3.5);
  \node [below] at (1.5,0) {\tiny $i$}; \node [below] at (2.5,0) {\tiny $\si(i)$}; \node [below] at (3.5,0) {\tiny $j$};
  \node [left]  at (0,1.5) {\tiny $i$}; \node [left]  at (0,2.5) {\tiny $\si(i)$}; \node [left]  at (0,3.5) {\tiny $j$}; \node [left] at (0,0) {\phantom{\tiny $\si(j)$}};
\end{tikzpicture}
\caption*{Type~3}
\bigskip
\end{subfigure}
\begin{subfigure}[b]{0.3\textwidth}
\begin{tikzpicture}[scale=0.55]
  \draw (0,0) rectangle (5,5);
  \draw[fill=lightgray] (1,3)--(2,3)--(2,4)--(1,4)--(1,3);
  \draw[fill=lightgray] (3,1)--(4,1)--(4,2)--(3,2)--(3,1);
  \draw[dotted] (0,0)--(5,5);
  \bvartex (s1) at (1,3) {}; \bvertex (v1) at (2,3) {}; \bvartex (s2) at (2,4) {}; \bvertex (v2) at (1,4) {};
  \bvartex (s3) at (3,1) {}; \bvertex (v3) at (4,1) {}; \bvartex (s4) at (4,2) {}; \bvertex (v4) at (3,2) {};
  \draw (1,0.15)--(1,-0.15); \draw (2,0.15)--(2,-0.15); \draw (3,0.15)--(3,-0.15); \draw (4,0.15)--(4,-0.15);
  \draw (0.15,1)--(-0.15,1); \draw (0.15,2)--(-0.15,2); \draw (0.15,3)--(-0.15,3); \draw (0.15,4)--(-0.15,4);
  \node [below] at (1,0) {\tiny $i$}; \node [below] at (2,0) {\tiny $j$}; \node [below] at (3,0) {\tiny $\si(i)$}; \node [below] at (4,0) {\tiny $\si(j)$};
  \node [left]  at (0,1) {\tiny $i$}; \node [left]  at (0,2) {\tiny $j$}; \node [left]  at (0,3) {\tiny $\si(i)$}; \node [left]  at (0,4) {\tiny $\si(j)$};
\end{tikzpicture}
\caption*{Type~4}
\end{subfigure}
\begin{subfigure}[b]{0.3\textwidth}
\begin{tikzpicture}[scale=0.55]
  \draw (0,0) rectangle (5,5);
  \draw[fill=lightgray] (1,2)--(2,2)--(2,1)--(4,1)--(4,3)--(3,3)--(3,4)--(1,4)--(1,2);
  \draw[dotted] (0,0)--(5,5);
  \bvartex (s1) at (1,2) {}; \bvertex (v1) at (2,2) {}; \bvartex (s2) at (2,1) {}; \bvertex (v2) at (4,1) {};
  \bvartex (s3) at (4,3) {}; \bvertex (v3) at (3,3) {}; \bvartex (s4) at (3,4) {}; \bvertex (v4) at (1,4) {};
  \draw (1,0.15)--(1,-0.15); \draw (2,0.15)--(2,-0.15); \draw (3,0.15)--(3,-0.15); \draw (4,0.15)--(4,-0.15);
  \draw (0.15,1)--(-0.15,1); \draw (0.15,2)--(-0.15,2); \draw (0.15,3)--(-0.15,3); \draw (0.15,4)--(-0.15,4);
  \node [below] at (1,0) {\tiny $i$}; \node [below] at (2,0) {\tiny $\si(i)$}; \node [below] at (3,0) {\tiny $j$}; \node [below] at (4,0) {\tiny $\si(j)$};
  \node [left]  at (0,1) {\tiny $i$}; \node [left]  at (0,2) {\tiny $\si(i)$}; \node [left]  at (0,3) {\tiny $j$}; \node [left]  at (0,4) {\tiny $\si(j)$};
\end{tikzpicture}
\caption*{Type~5}
\end{subfigure}
\begin{subfigure}[b]{0.3\textwidth}
\begin{tikzpicture}[scale=0.55]
  \draw (0,0) rectangle (5,5);
  \draw[fill=lightgray] (1,2)--(2,2)--(2,1)--(3,1)--(3,2)--(4,2)--(4,3)--(3,3)--(3,4)--(2,4)--(2,3)--(1,3)--(1,2);
  \draw[dotted] (0,0)--(5,5);
  \bvartex (s1) at (1,2) {}; \bvartex (s2) at (2,1) {}; \bvertex (v1) at (3,1) {}; \bvertex (v2) at (4,2) {};
  \bvartex (s3) at (4,3) {}; \bvartex (s4) at (3,4) {}; \bvertex (v3) at (2,4) {}; \bvertex (v4) at (1,3) {};
  \draw (1,0.15)--(1,-0.15); \draw (2,0.15)--(2,-0.15); \draw (3,0.15)--(3,-0.15); \draw (4,0.15)--(4,-0.15);
  \draw (0.15,1)--(-0.15,1); \draw (0.15,2)--(-0.15,2); \draw (0.15,3)--(-0.15,3); \draw (0.15,4)--(-0.15,4);
  \node [below] at (1,0) {\tiny $i$}; \node [below] at (2,0) {\tiny $\si(i)$}; \node [below] at (3,0) {\tiny $\si(j)$}; \node [below] at (4,0) {\tiny $j$};
  \node [left]  at (0,1) {\tiny $i$}; \node [left]  at (0,2) {\tiny $\si(i)$}; \node [left]  at (0,3) {\tiny $\si(j)$}; \node [left]  at (0,4) {\tiny $j$};
\end{tikzpicture}
\caption*{Type~6}
\end{subfigure}
\caption{The involutions $\si$ and $\tau=\ct_{(i,j)}(\si)$.
The black and white dots mark the positions of the elements of $\si$ and $\tau$, respectively, which ``move'' in the transformation.
Inside the grey areas, there are no elements of $\si$.} \label{Incitti}
\end{figure}
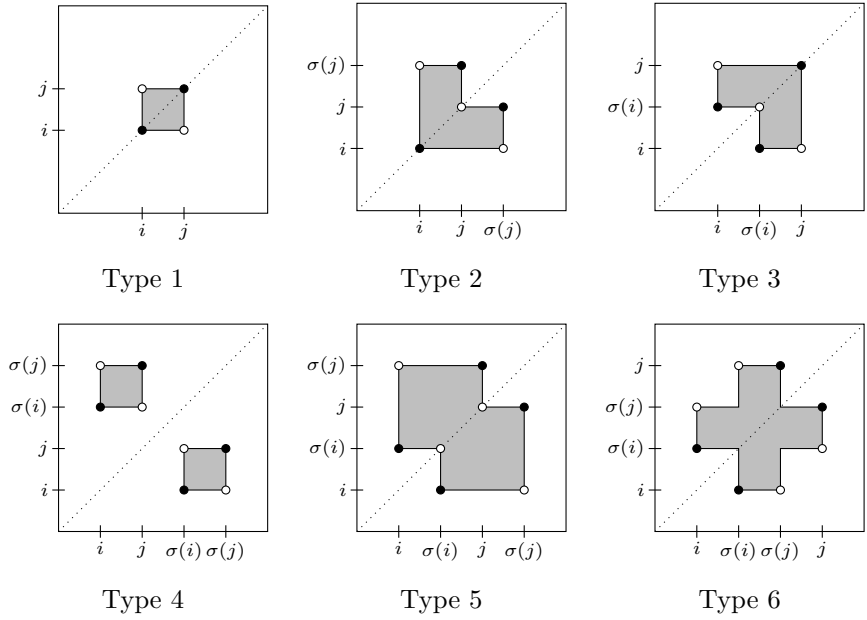

Incitti characterised the covering relation for the involutions as follows.

\begin{Lemma}[\mbox{\cite[\Th~5.1]{Incitti_A}}] \label{cover}
Let $\si,\tau \in \I$. Then $\si \cov \tau$ in $\I$ if and only if $\tau=\ct_{(i,j)}(\si)$ for some (necessarily unique) pair $(i,j)$.
\end{Lemma}

For $\si<\tau$ in $\I$, define $\di=\di(\si,\tau)=\min{\{i \mid \si(i) \neq \tau(i)\}}$ and $\ci=\ci(\si,\tau)=\min{\{j \geq {\di}+1 \mid \si(j) \in [\si(\di)+1,\tau(\di)]\}}$,
where $\di$ and $\ci$ are short for \emph{difference index} and \emph{covering index}, respectively.
It is clear from Figure \ref{Incitti} that if $\si \cov \tau$, then $i=\di(\si,\tau)$ and $j=\ci(\si,\tau)$.
If $\si<\tau$, we sometimes write $\mct_\tau(\si)$ instead of $\ct_{(\di,\ci)}(\si)$.
Here $\mathrm{m}$ stands for \emph{minimal}: $(\di,\ci)$ is lex-minimal among all pairs $(i,j)$ such that $\ct_{(i,j)}(\si) \leq \tau$.

Let us now consider the signed involutions.
Here we have the following characterisation of the covering relation.
It follows from \cite[\Th~3.6, \Def~4.2, and Remark~4.3]{Incitti_B}.
The necessary parts of the transformation $\Bct$ will be explained below.

\begin{Lemma}[\cite{Incitti_B}] \label{Bcover}
Let $\si,\tau \in \IB$. Then $\si \cov \tau$ in $\IB$ if and only if $\tau=\Bct_{(i,j)}(\si)$ for some (necessarily unique) pair $(i,j)$.
\end{Lemma}

If $\tau=\Bct_{(i,j)}(\si)$ for some pair $(i,j)$, let $\la(\si,\tau)=(i,j)$.
By Lemma~\ref{Bcover}, this defines an edge-labelling of $\IB$, with $\{(i,j) \in [\pm n]^2 \mid i<j\}$ totally ordered by the lexicographic order.
We call $(i,j)$ the \emph{label} on a cover $\si \cov \tau$ if $\la(\si,\tau)=(i,j)$, and a label on a chain if it is the label on some cover of the chain.

In order to establish EL-shellability of the fixed point free signed involutions, we need the following results of Incitti.

\begin{Lemma}[\mbox{\cite[\Th~4.4]{Incitti_B}}] \label{incr}
Let $\si<\tau$ in $\IB$. Then there is exactly one increasing $\si$-$\tau$-chain, and it is lex-minimal.
\end{Lemma}

\begin{Lemma}[\mbox{\cite[\Th~4.6]{Incitti_B}}] \label{decr}
Let $\si<\tau$ in $\IB$. Then there is exactly one decreasing $\si$-$\tau$-chain.
\end{Lemma}

\begin{Remark}
Since $\Bct_{(i,j)}(\si)(i)>\Bct_{(i,j)}(\si)(j)$, there is also exactly one $\si$-$\tau$-chain with weakly decreasing labels.
\end{Remark}

We now delve somewhat deeper into the covering relation for the signed involutions.
For $\si<\tau$ in $\IB$, Incitti defines a new signed involution $\pi=\mBct_\tau(\si)$, where $\mathrm{m}$ stands for \emph{minimal} (this is explained below), such that $\si \cov \pi \leq \tau$.
This is done in \cite[Tables~2--9]{Incitti_B}, which correspond to eight different cases, two of which are called easy, four normal, and two hard.
For example, N2.2 in Table~5 is normal of the second kind, and the first move performed to reach $\pi$ from $\si$, in the sense of Figure~\ref{Incitti}, is of type~2.
Black dots are used for $\si$, white dots for $\pi$, and white squares for $\tau$.
The light grey regions correspond to the first move that we perform.
Consider, e.g., N2.2 where the black dots move according to type~2 to form the involution $\mct_\tau(\si)$ (note that this is not a signed involution).
The medium and dark grey regions correspond to the second and third moves.

The top right cells of Tables~2--9 explain further how to interpret the permutation diagrams.
The number $\Bcv(\si,\tau)$, where $\cv$ is short for \emph{covering value}, is defined as $\pi(\di)$.
The exact definition of $\Bci$ is case dependent and not important here.
However, it turns out that if $\tau=\Bct_{(i,j)}(\si)$, then $i=\di$ and $j=\Bci$.
Moreover, if $\si<\tau$, then $(\di,\Bci)$ is lex-minimal among all pairs $(i,j)$ such that $\Bct_{(i,j)}(\si) \leq \tau$.

The reader of this paper can safely ignore the quintuple associated with each diagram, as well as the information given in the top left and top middle cells.

We are now ready to prove EL-shellability of the fixed point free signed involutions.
As mentioned earlier, we follow the strategy in the proof of \cite[\Th~1.4]{Hansson_graded_A}.
That is, we prove that the decreasing $\si$-$\tau$-chain in $\IB$ is contained in $\FB$ and that this chain is lex-maximal.
The major difference between the two proofs is that we now have to consider many more possible coverings.
Moreover, some arguments have been simplified.

\begin{Thm} \label{thm:EL}
The poset $\FB$ is EL-shellable.
\end{Thm}

\begin{proof}
Let $\si<\tau$ in $\FB$, let $C_D=(\si \cov \tau_k \cov \cd \cov \tau_1 \cov \tau)$ be the decreasing $\si$-$\tau$-chain in $\IB$, where $k \geq 1$,
and let $(i_\tau,j_\tau)$ be the label on $\tau_1 \cov \tau$.
Observe that if we can prove that $\tau_1 \in \FB$, it follows that $\tau_1,\p,\tau_k \in \FB$, because the decreasing $\si$-$\tau_1$-chain in $\IB$ is $\si \cov \tau_k \cov \cd \cov \tau_2 \cov \tau_1$.

Let $h=\di(\si,\tau)$.
Since $\si$ has no fixed points, $h$ is an exceedance of $\si$ (i.e., $\si(h)>h$).
Indeed, if $\si(i)<i$, then $\si(i)$ is an exceedance of $\si$, and $\si(\si(i))=\tau(\si(i)) \ToT \si(i)=\tau(i)$.
Consider any cover $\pi \cov \pi'$ in $[\si,\tau]$.
Two applications of \Def~\ref{order} yield $\di(\pi,\pi') \geq \di(\pi,\tau) \geq h$.
Since the label $(i,j)$ on $\pi \cov \pi'$ satisfies $i=\di(\pi,\pi')$, it follows that $h$ is in some label on $C_D$, and because $C_D$ is decreasing, $i_\tau=h$.

Suppose $\tau_1 \notin \FB$ so that $\tau_1 \cov \tau$ is a case where the number of fixed points decreases.
The only such cases are E1.1, N1.1, N4.2, and H2.1, and in all of them, $h$ is a fixed point of $\tau_1$.
Observe that $\tau_1$ has fixed points also in E2.6a, N1.2a, N1.2b, N1.3a, N1.3b, N2.2, H1.2, H2.2a, H2.2b, and H2.3, and $h$ is not always one of them, but the number of fixed points does not decrease in those cases.
Therefore, $\si[h,h+1]>\tau_1[h,h+1]$.
Since $\si \leq \tau_1$, it follows from \Def~\ref{order} that $\tau_1 \in \FB$.

We have proved that the decreasing $\si$-$\tau$-chain in $\IB$ is contained in $\FB$.
The next step is to show that this chain is lex-maximal.
This is done in the same way as in the proof of \cite[\Th~1.4]{Hansson_graded_A}, where Incitti's type~$A$ version of Lemma~\ref{incr} is used.
We include the argument here in order to make the proof complete.

To obtain a contradiction, consider the lex-maximal $\si$-$\tau$-chain in $\IB$, and suppose it is not decreasing;
call it $C=(\pi_1 \cov \cd \cov \pi_{k+2})$ and say that $\la(\pi_1,\pi_2) \leq \la(\pi_2,\pi_3)$.
By Lemma~\ref{incr}, $\pi_1 \cov \pi_2 \cov \pi_3$ is lex-minimal among the $\pi_1$-$\pi_3$-chains in $\IB$.
Thus, $\pi_1 \cov \pi_2' \cov \pi_3 \cov \cd \cov \pi_{k+2}$, where $\pi_1 \cov \pi_2' \cov \pi_3$ is the decreasing $\pi_1$-$\pi_3$-chain, is lex-larger than $C$, a contradiction.

So, the decreasing $\si$-$\tau$-chain in $\IB$ is contained in $\FB$ and it is lex-maximal.
Now, by reversing the lexicographic order on the set $\{(i,j) \in [\pm n]^2 \mid i<j\}$, we obtain an edge-labelling of $\FB$ such that in each interval $[\si,\tau]$, there is an increasing $\si$-$\tau$-chain which is lex-minimal.
By Lemma~\ref{decr} and the remark following it, this is an EL-labelling of $\FB$.
\end{proof}

The fact that the decreasing chain in $\IB$ is contained in $\FB$ allows as to compute the dimension of the PL ball $\De(\FB-\{\hat{0},\hat{1}\})$ as
$\rho(\hat{1})-\rho(\hat{0})-2$, where $\rho$ is the rank function in $\IB$.
In $\FB$, $\hat{1}$ is the reverse permutation $w_0$ and $\hat{0}$ is the product $(-n,-n+1)(-n+2,-n+3) \cd (n-1,n)$ of adjacent transpositions.

The length of $\si$ in $S_n^B$ is
\[
\ell(\si)=\frac{\inv(\si)+\negg(\si)}{2},
\]
where
\[
\inv(\si)=|\{(i,j) \in [\pm n]^2 \mid \h{$i<j$ and $\si(i)>\si(j)$}\}|
\]
(the number of inversions of $\si$, i.e., the length of $\si$ in $S_{2n}$) and
\[
\negg(\si)=|\{i \in [n] \mid \si(i)<0\}|.
\]
By \cite[\Th~3.7]{Incitti_B},
\[
\rho(\si)=\frac{\ell(\si)+\dna(\si)}{2}
\]
where
\[
\dna(\si)=|\{i \in [n] \mid -i \leq \si(i)<i\}|.
\]
It follows that $\rho(\hat{1})=\frac{n^2+n}{2}$.
Furthermore, $\rho(\hat{0})=\frac{n}{2}$ if $n$ is even and $\frac{n+1}{2}$ if $n$ is odd, so the dimension is
\[
\begin{cases}\frac{n^2}{2}-2 & \h{if $n$ is even,} \\[1mm] \frac{n^2-1}{2}-2 & \h{if $n$ is odd.}\end{cases}
\]

\bibliographystyle{amsplain}

\bibliography{Referenser_bara_initialer}

\end{document}